\documentclass[]{article}

\usepackage[utf8]{inputenc}
\usepackage[margin=1.25in]{geometry}

\usepackage{amsthm, amsfonts, amsmath, makecell, tikz, 
 graphicx, colortbl}
\usepackage{multirow, booktabs, bigstrut}

\usepackage{mathtools}

\usepackage{graphicx}
\usepackage{float}
\usepackage{color}
\usepackage{tikz,pgfplots}
\usepackage{tikz-cd}
\usepackage{verbatim}
\usepackage{enumerate}
\usepackage{physics}
\usepackage{braket}

\usepackage{bm}

\usepackage{enumerate, amssymb}

\usetikzlibrary{graphs,graphs.standard,arrows.meta,calc,intersections}

\usepackage[]{algorithm, algorithmic}

\usepackage{tabularx}

\theoremstyle{plain}
\newtheorem{theorem}{Theorem}[section]
\newtheorem{lemma}[theorem]{Lemma}
\newtheorem{corollary}[theorem]{Corollary}
\newtheorem{proposition}[theorem]{Proposition}

\newtheorem{remark}[theorem]{Remark}

\theoremstyle{definition}
\newtheorem{definition}[theorem]{Definition}
\newtheorem{example}[theorem]{Example}

\DeclareMathOperator{\Int}{Int}

\DeclareMathOperator{\Sk}{{Sk}}

\DeclareMathOperator{\IntR}{Int{}^\text{R}}

\DeclareMathOperator{\minval}{minval}
\DeclareMathOperator{\loc}{loc}

\DeclareMathOperator{\Spec}{Spec}

\renewcommand{\epsilon}{\varepsilon}

\newcommand{\N}{{\mathbb N}}
\newcommand{\C}{{\mathbb C}}
\newcommand{\Q}{{\mathbb Q}}

\newcommand{\Z}{{\mathbb Z}}

\newcommand{\p}{\mathfrak{p}}
\newcommand{\q}{\mathfrak{q}}
\newcommand{\m}{\mathfrak{m}}
\newcommand{\M}{\mathfrak{M}}
\renewcommand{\P}{\mathfrak{P}}

\DeclareMathOperator{\Max}{Max}

\definecolor{gray}{rgb}{.5,.5,.5}

\definecolor{black}{rgb}{0,0,0}

\definecolor{blue}{rgb}{0,0,1}

\definecolor{red}{rgb}{1,0,0}

\definecolor{green}{rgb}{0,1,0}

\definecolor{gold}{rgb}{.5,.5,.2}

\definecolor{yellow}{rgb}{1,1,.4}

\definecolor{purple}{rgb}{.5,0,.5}

\definecolor{darkgreen}{rgb}{0,.5,0}

\definecolor{orange}{rgb}{1,.55,0}

\definecolor{white}{rgb}{1,1,1}

\usepackage[colorinlistoftodos]{todonotes}

\let\originalleft\left
\let\originalright\right
\renewcommand{\left}{\mathopen{}\mathclose\bgroup\originalleft}
\renewcommand{\right}{\aftergroup\egroup\originalright}

\let\originaltodo\todo
\renewcommand{\todo}[1]{\originaltodo[inline]{#1}}

\usepackage{hyperref}
\usepackage{xcolor}
\hypersetup{
	colorlinks,
	linkcolor={red!50!black},
	citecolor={blue!50!black},
	urlcolor={blue!80!black}
}

\usepackage{subfigure}

\title{The Skolem property in rings of integer-valued rational functions}

\author{Baian Liu}

\begin{document}

\maketitle

\begin{abstract}
	Let $D$ be a domain and let $\Int(D)$ and $\IntR(D)$ be the ring of integer-valued polynomials and the ring of integer-valued rational functions, respectively. Skolem proved that if $I$ is a finitely-generated ideal of $\Int(\Z)$ with all the value ideals of $I$ not being proper, then $I = \Int(\Z)$. This is known as the Skolem property, which does not hold in $\Z[x]$. One obstruction to $\Int(D)$ having the Skolem property is the existence of unit-valued polynomials. This is no longer an obstruction when we consider the Skolem property on $\IntR(D)$. We determine that the Skolem property on $\IntR(D)$ is equivalent to the maximal spectrum being contained in the ultrafilter closure of the set of maximal pointed ideals. We generalize the Skolem property using star operations and determine an analogous equivalence under this generalized notion. 
\end{abstract}

\section{Introduction}

	\indent\indent Given a domain $D$, the ring of integer-valued polynomials over $D$ has been studied extensively. A collection of results on integer-valued polynomials can be found in \cite{Cahen}. However, not much is known about the ring of integer-valued rational functions. Despite having similar definitions, the ring of integer-valued polynomials and the ring of integer-valued rational functions can behave very differently. We start by giving the definitions of these two ring extensions of $D$ with a variation that is slightly more general. 

\begin{definition}
	Let $D$ be a domain with $K$ as the field of fractions. Also let $E \subseteq K$ be a non-empty subset. We denote by
	\[
		\IntR(D) \coloneqq \{ \varphi \in K(x) \mid \varphi(a) \in D,\, \forall a \in D \} \quad \text{and} \quad	\IntR(E,D) \coloneqq \{ \varphi \in K(x) \mid \varphi(a) \in D,\,\forall a \in E \}
	\]
	the \textbf{ring of integer-valued rational functions over $D$} and the \textbf{ring of integer-valued rational functions on $E$ over $D$}, respectively.	Note that $\IntR(D,D) = \IntR(D)$. 
	
	Compare these definitions to 
	\[
		\Int(D) \coloneqq \{ f \in K[x]\mid f(a) \in D,\, \forall a \in D \} \quad \text{and} \quad	\Int(E,D) \coloneqq \{ f \in K[x] \mid f(a) \in D,\,\forall a \in E \}
	\]
	the \textbf{ring of integer-valued polynomials over $D$} and \textbf{ring of integer-valued polynomials on $E$ over $D$}, respectively. 
\end{definition}

One way the behavior of rings of integer-valued rational functions differs from that of integer-valued polynomials is the Skolem property. Having the Skolem property allows us to determine if a finitely-generated ideal of $\IntR(E,D)$ is proper through evaluation.  

\begin{definition}
	Let $D$ be a domain with field of fractions $K$. Take $E \subseteq K$ to be some subset. Let $A$ be a ring such that $D \subseteq A \subseteq \IntR(E,D)$. For $I$ an ideal of $A$ and $a \in E$, we denote by
	\[
	I(a) \coloneqq \{\varphi(a) \mid \varphi \in I \}
	\]
	the \textbf{value ideal of $I$ at $a$}. Note that $I(a)$ is an ideal of $D$. We say that $A$ has the \textbf{Skolem property} if for every finitely-generated ideal $I$ of $\IntR(E,D)$ such that $I(a) = D$ for all $a \in E$, then $I = A$. 
	
\end{definition}

In order for $\Int(D)$ to have the Skolem property, it is necessary that there does not exist a non-constant $f \in D[x]$ such that $f(d) \in D^\times$ for all $d \in D$. Such a polynomial $f$ is called a \textbf{unit-valued polynomial}. This is because all of the value ideals for $f\Int(D)$ equal $D$, but $f$ is not a unit of $\Int(D)$. We know that $f$ is not a unit of $\Int(D)$ because $\frac{1}{f}$ is not a polynomial. However, $\frac{1}{f}$ is a rational function, so $f\IntR(D) = \IntR(D)$. Therefore, the existence of a unit-valued polynomial prevents $\Int(D)$ from having the Skolem property, but it does not necessarily prevent $\IntR(D)$ from having the Skolem property. 

Furthermore, if $D$ is a noetherian domain that is not a field, then $\Int(D)$ having the Skolem property implies that for every maximal ideal $\m$ of $D$, we have that $D/\m$ is algebraically closed, or $\m$ has height one and a finite residue field \cite[2.1 Proposition]{ChabertPolynomesAValeursEntieres}. For $\IntR(D)$, the Skolem property is not as restrictive. For example, if $D$ is any local domain with a finite residue field, then $\IntR(D)$ has the Skolem property \cite[Proposition X.3.7]{Cahen}. Another example that the Skolem property is less restrictive for rings of integer-valued rational functions is that if $D$ is a domain with $E$ a non-empty subset of the field of fractions of $D$ such that $\IntR(E,D)$ is a Prüfer domain, then $\IntR(E,D)$ has the Skolem property \cite[Lemma 4.1]{IntValuedRational}. One condition that makes the ring $\IntR(D)$ Prüfer is if $D$ is a Prüfer domain such that there exists a monic unit-valued polynomial in $D[x]$ \cite[Corollary 3.4]{PruferNonDRings}, which can be achieved with no restrictions on the height of maximal ideals of $D$ using \cite[Corollary 2.6]{PruferNonDRings}. 

The Skolem property is also related to the form of the maximal ideals. Given a noetherian domain $D$ of dimension strictly greater than one, we have that $\Int(D)$ has the Skolem property if and only if all of the maximal ideals of $\Int(D)$ is of the form $\{f(x) \in \Int(D) \mid f(a) \in \m\hat{D}_\m \}$ for some maximal ideal $\m$ of $D$ and $a \in \hat{D}_\m$, where $\hat{D}_\m$ is the $\m$-adic completion of $D$ \cite[4.6 Théorème]{ChabertPolynomesAValeursEntieres}. A result of the same nature concerning rings integer-valued rational functions is that if $V$ is a valuation domain with $E$ a non-empty subset of the field of fractions such that $\IntR(E,V)$ is Prüfer, then all of the maximal ideals of $\IntR(E,V)$ are ultrafilter limits of the set of maximal pointed ideals \cite[Proposition 6.1]{IntValuedRational}. We will define ultrafilter limits and maximal pointed ideals in the next section. 

In this work, we generalize the relationship between the Skolem property and the maximal ideals of rings of integer-valued rational functions. Section \ref{Sect:Background} gives definitions and previous results necessary to state our results. Section \ref{Sect:SkolemSpectrum} connects the Skolem property and the maximal spectrum of a ring of integer-valued rational functions. Section \ref{Sect:StarStar} extends this connection using star operations. Section \ref{Sect:StrongSkolem} provides examples of when the strong Skolem property, a property for distinguishing finitely-generated ideals through evaluation, fails.

\section{Background}\label{Sect:Background}

\indent\indent A notion related to the Skolem property is the strong Skolem property. The strong Skolem property can be thought of as the ability to distinguish finitely-generated ideals of $\IntR(E,D)$ using evaluation. 

\begin{definition}
	Suppose $D$ is a domain with field of fractions $K$ and $E$ is a subset of $K$. Then we say that $\IntR(E,D)$ has the \textbf{strong Skolem property} if whenever $I$ and $J$ are finitely-generated ideals of $\IntR(E,D)$ such that $I(a) = J(a)$ for all $a \in E$, then $I = J$. 
\end{definition}

\begin{remark}
	The strong Skolem property implies the Skolem property by taking $J = \IntR(E,D)$. 
\end{remark}

The strong Skolem property is closely related to the property of being a Prüfer domain. If $\IntR(D)$ is a Prüfer domain, then $\IntR(D)$ has the strong Skolem property \cite[Theorem 4.2]{IntValuedRational}. In fact, if $D$ is assumed to be a Prüfer domain, then $\IntR(D)$ having the strong Skolem property implies that $\IntR(D)$ is a Prüfer domain \cite[Theorem 4.7]{Steward}.

We may reformulate the definitions of the Skolem property and the strong Skolem property using a closure operation. 

\begin{definition}
	Let $D$ be a domain and let $E$ be a subset of the field of fractions of $D$. For an ideal $I$ of $\IntR(E,D)$, the \textbf{Skolem closure of $I$} is 
	\[
		I^{\Sk} \coloneqq \{\varphi \in \IntR(E,D) \mid \varphi(a) \in I(a) \text{ for all $a \in E$} \}. 
	\]
	We say that $I$ is \textbf{Skolem closed} if $I = I^{\Sk}$. 
\end{definition}

\begin{remark}
	Take $D$ to be a domain with field of fractions $K$. Let $E$ be a subset of $K$. Then 
	\begin{enumerate}
		\item $\IntR(E,D)$ has the Skolem property if and only if the Skolem closure of a finitely-generated proper ideal of $\IntR(E,D)$ is a proper ideal, and
		\item $\IntR(E,D)$ has the strong Skolem property if and only if every finitely-generated ideal of $\IntR(E,D)$ is Skolem closed. 
	\end{enumerate}
\end{remark}

In Section \ref{Sect:StrongSkolem}, we give an example of a domain $D$ such that $\IntR(D)$ has the Skolem property but not the strong Skolem property.

In this work, we relate the Skolem property and related properties to the prime spectrum of $\IntR(E,D)$. The following are some definitions concerning the prime spectrum of $\IntR(E,D)$.

\begin{definition}
	Let $D$ be a domain with field of fractions $K$. Take $E$ to be some subset of $K$. For any $a \in E$ and $\p \in \Spec(D)$, we can define
	\[
		\P_{\p, a} \coloneqq \{\varphi \in \IntR(E,D) \mid \varphi(a) \in \p \}.
	\]
	The set $\P_{\p, a}$ is a prime ideal of $\IntR(E,D)$ since $\IntR(E,D)/\P_{\p,a} \cong D/\p$. Prime ideals of $\IntR(E,D)$ of the form $\P_{\p, a}$ for some $a \in E$ and $\p \in \Spec(D)$ are called \textbf{pointed prime ideals}. If $\m$ is a maximal ideal of $D$, then we write $\M_{\m,a}$ for $\P_{\m,a}$. Ideals of $\IntR(E,D)$ of the form $\M_{\m,a}$ for some $a \in E$ and $\m \in \Max(D)$, where $\Max(D)$ is the set of all maximal ideals of $D$, are called \textbf{maximal pointed ideals}. The maximal pointed ideals are indeed maximal ideals of $\IntR(E,D)$.
\end{definition}

Pointed prime ideals and maximal pointed ideals in general do not describe all of the prime ideals or all of the maximal ideals of a ring in the form $\IntR(E,D)$. Some prime ideals of $\IntR(E,D)$ are described through ultrafilters. 

\begin{definition}
	Let $S$ be a set. A \textbf{filter} $\mathcal{F}$ on $S$ is a collection of subsets of $S$ such that
	\begin{enumerate}
		\item $\emptyset \notin \mathcal{F}$;
		\item if $A, B \in \mathcal{F}$, then $A \cap B \in \mathcal{F}$; and
		\item if $A \in \mathcal{F}$ and $B \subseteq S$ is such that $A \subseteq B$, then $B \in \mathcal{F}$. 
	\end{enumerate}
	If $\mathcal{U}$ is a filter on $S$ such that for every $A \subseteq S$, we have $A \in \mathcal{U}$ or $S \setminus A \in \mathcal{U}$, then we call $\mathcal{U}$ an \textbf{ultrafilter}. Every filter of $S$ is contained in some ultrafilter of $S$ by the Ultrafilter Lemma. 
\end{definition}

Through the use of filters and ultrafilters, we can obtain a notion of a limit of a family of ideals.

\begin{definition}
	Let $R$ be a commutative ring. Take $\{I_\lambda\}_{\lambda \in \Lambda}$ to be a family of ideals of $R$. For each $r \in R$, we define the \textbf{characteristic set of $r$ on $\{I_\lambda\}$} to be
	\[
		\chi_r = \{I_\lambda \mid r \in I_\lambda \}.
	\]
	For a filter $\mathcal{F}$ on $\{I_\lambda\}$, we define the \textbf{filter limit of $\{I_\lambda\}$ with respect to $\mathcal{F}$} as
	\[
		\lim\limits_{\mathcal{F}} I_\lambda = \{r \in R \mid \chi_r \in \mathcal{F} \}.
\]
If $\mathcal{F}$ is an ultrafilter, we call $\lim\limits_{\mathcal{F}} I_\lambda$ the \textbf{ultrafilter limit of $\{I_\lambda\}$ with respect to $\mathcal{F}$}.
\end{definition}

\begin{remark}
	The filter and ultrafilter limits of a family of ideals are also themselves ideals. If $\{\p_\lambda\}_{\lambda \in \Lambda}$ is a family of prime ideals and $\mathcal{U}$ is an ultrafilter of $\{\p_\lambda\}$, then the ultrafilter limit $\lim\limits_{\mathcal{U}} \p_\lambda$ is also a prime ideal. This gives rise to the \textbf{ultrafilter topology} on $\Spec(R)$, which is identical to the patch topology and the constructible topology on $\Spec(R)$ \cite{UltrafilterTopology}. 
\end{remark}

We now define star operations gives some examples and properties related to star operations. This will later assist us in extending the Skolem property. Furthermore, under certain conditions, we can derive a star operation from the Skolem closure. 

\begin{definition}
	Let $D$ be a domain with field of fractions $K$. Denote by $F(D)$ the set of nonzero fractional ideals of $D$. A map $F(D) \to F(D)$ given by $I \mapsto I^\star$ is a \textbf{star operation on $D$} if for every $a \in K$ and all $I, J \in F(D)$, we have
	\begin{enumerate}
		\item $(a)^\star = (a)$, $(aI)^\star = aI^\star$;
		\item $I \subseteq I^\star$, if $I \subseteq J$, then $I^\star \subseteq J^\star$; and
		\item $(I^\star)^\star = I^\star$
	\end{enumerate}
	
	We say that $I$ is a \textbf{$\star$-ideal} if $I^\star = I$. We say two star operations $\star_1$ and $\star_2$ on $D$ are equal, denoted as $\star_1 = \star_2$ if for all $I \in F(D)$, we have $I^{\star_1} = I^{\star_2}$. We say that $\star_1 \leq \star_2$ if for all $I \in F(D)$, we have $I^{\star_1} \subseteq I^{\star_2}$. 
\end{definition}

\begin{example}
	Let $D$ be a domain. The $d$-operation given by $I_d = I$ for all nonzero fractional ideals of $D$ is a star operation. 
\end{example}

\begin{example}\cite{WangMcCasland}
	Let $D$ be a domain with field of fractions $K$. A nonzero fractional ideal $J$ of $D$ is called a GV-ideal if $J$ is finitely-generated and $J^{-1} = R$. The $w$-operation is a star operation given by
	\[
		I_w = \{a \in K \mid \text{there exists a GV-ideal $J$ such that $aJ \subseteq I$}  \}
	\]
	for each $I \in F(D)$. 
\end{example}

\begin{example}
	Let $D$ be a domain. Take a nonzero fractional ideal $I$ of $D$. The divisorial closure operation $v$ given by $I_v = (I^{-1})^{-1}$ is a star operation. Alternatively, $I_v$ is the intersection of the nonzero principal fractional ideals of $D$ containing $I$. 
\end{example}

\begin{example}
	Let $D$ be a domain and $I$ be a nonzero fractional ideal of $D$. The $t$-operation given by
	\[
		I_t = \bigcup_{J \in \Lambda_I} J_v,
	\]
	where $\Lambda_I$ is the family of all nonzero finitely-generated fractional ideals of $D$ contained in $I$, is a star operation. 
\end{example}

\begin{remark}
	For any domain $D$, we have $d \leq w \leq t \leq v$ \cite{Mimouni}. Moreover, for any star operation $\star$ on $D$, we have $\star \leq v$. 
\end{remark}

In general, we can do the same procedure as the one to obtain the $t$-operation from the $v$-operation on any star operation to obtain another star operation.

\begin{definition}
	Let $D$ be a domain and $\star$ be a star operation on $D$. We define the star operation $\star_f$ on $D$ by taking a nonzero fractional ideal $I$ of $D$ and assigning
	\[
		I^{\star_f} = \bigcup_{J \in \Lambda_I} J^\star,
	\]
	where $\Lambda_I$ is the family of all nonzero finitely-generated fractional ideals of $D$ contained in $I$. Note that $\star_f \leq \star$. 
	
	If $\star = \star_f$, we say that $\star$ is a \textbf{finite type} star operation.
\end{definition}

\begin{example}
	The $t$-operation is a finite type star operation. In fact, for any finite type star operation $\star$ on a domain $D$, we have $\star \leq t$. 
\end{example}

Finite type star operations are important as they behave nicely with filter limits.
\begin{proposition}\label{Prop:FiltersPreserveStar}
	Let $\star$ be a finite type star operation on $D$. Suppose that $\lim\limits_{\mathcal{F}} \mathfrak{a}_\lambda$ is the limit of a family $\{\mathfrak{a}_\lambda\}$ of $\star$-ideals of $D$ with respect to a filter $\mathcal{F}$. If $\lim\limits_{\mathcal{F}} \mathfrak{a}_\lambda$ is a nonzero ideal, then $\lim\limits_{\mathcal{F}} \mathfrak{a}_\lambda$ is a $\star$-ideal.
\end{proposition}

\begin{proof}
 Let $J \subseteq \lim\limits_{\mathcal{F}} \mathfrak{a}_\lambda$ be a nonzero finitely-generated ideal, generated by $a_1, \dots, a_n$. Let $\mathfrak{a}_\lambda \in \chi_{a_1} \cap \cdots \cap \chi_{a_n}$. This means that $J \subseteq \mathfrak{a}_\lambda$ and since $\mathfrak{a}_\lambda$ is a $\star$-ideal, we have that $J^\star \subseteq \mathfrak{a}_\lambda$. Thus, $\chi_{a_1} \cap \cdots \cap \chi_{a_n} \subseteq \chi_a$ for all $a \in J^\star$, which implies that $\chi_a \in \mathcal{F}$ for all $a \in J^\star$. This shows that $J^\star \subseteq \lim\limits_{\mathcal{F}} \mathfrak{a}_\lambda$. Since $\star$ is of finite type, we can conclude that $\left(\lim\limits_{\mathcal{F}} \mathfrak{a}_\lambda\right)^\star = \lim\limits_{\mathcal{F}} \mathfrak{a}_\lambda$. 
\end{proof}

Given a domain $D$, its field of fractions $K$, and a subset $E$ of $K$, it is not true in general that the Skolem closure $\Sk$ on $\IntR(E, D)$ is a star operation. For example, let $D = F[t]$, where $F$ is any field, and let $E = F$. Consider the ideal $I = (x^2)$ in $\IntR(E,D)$. We know that $I(a) = (0)$ if $a = 0$ and $I(a) = D$ if $a \in E \setminus \{0\}$. If we take $\varphi = x$, then $\varphi(a) \in I(a)$ for all $a \in E$. However, $x \notin (x^2)$. Otherwise, $x = x^2 \psi$ for some $\psi \in \IntR(E,D)$ and then $\frac{1}{x} \in \IntR(E,D)$, a contradiction since $\frac{1}{x}$ is undefined evaluated at $x = 0$. This shows that the principal ideal $(x^2)$ is not Skolem closed, so $\Sk$ is not a star operation in this case. However, notice that in this example that for $\psi = \frac{1}{x}$, we have $\psi(a) \in I(a)$ for all $a \in E \setminus \{0\}$. The obstruction is only when we evaluate $\frac{1}{x}$ at $x = 0$. Thus, we should consider sets robust enough to handle finitely many obstructions. 

\begin{definition}\cite{CahenChabertSkolem}
	Let $D$ be a domain with field of fractions $K$. A subset $E$ of $K$ is a \textbf{strongly coherent set} for $\IntR(E,D)$ if for every finitely-generated ideal $I$ of $\IntR(E,D)$ and every rational function $\varphi \in K(x)$ such that $\varphi(a) \in I(a)$ for all but finitely many $a \in E$, we have $\varphi \in I^{\Sk}$. 
\end{definition}

For the purposes of the Skolem and strong Skolem properties, we are only considering finitely-generated ideals. It turns out if we apply the finite type construction to $\Sk$, we get a finite type star operation, which behaves nicely with filter limits according to Proposition \ref{Prop:FiltersPreserveStar} and is still descriptive of the Skolem and strong Skolem properties. 

\begin{definition}
	Let $D$ be a domain with $E$ a subset of the field of fractions. For a nonzero integral ideal $I$ of $\IntR(E,D)$, we define
	\[
	I^{\Sk_f} = \bigcup_{J \in \Lambda_I} J^{\Sk},
	\]
	where $\Lambda_I$ is the family of all nonzero finitely-generated ideals of $\IntR(E,D)$ contained in $I$. The operation $\Sk_f$ extends uniquely to the nonzero fractional ideals of $\IntR(E,D)$.
\end{definition}

The following proof is similar to that of \cite[Theorem 4.2]{Steward}.
\begin{proposition}
	Let $D$ be a domain with field of fractions $K$. Let $E \subseteq K$ such that $E$ is a strongly coherent set for $\IntR(E,D)$. Then $\Sk_f$ is a finite-type star operation on $\IntR(E,D)$. 
\end{proposition}

\begin{proof}
	It suffices to check the properties of a star operation on integral ideals \cite[p. 393]{Gilmer}. We begin by letting $\varphi \in K(x)$ be a nonzero element and $I$ be a nonzero finitely-generated integral ideal of $\IntR(E,D)$ and showing that $(\varphi I)^{\Sk} = \varphi I^{\Sk}$. We immediately have $\varphi I^{\Sk} \subseteq (\varphi I)^{\Sk}$. Now let $\psi \in (\varphi I)^{\Sk}$. Then for all $a \in E$ except the finitely many zeroes of $\varphi$, we have $\psi(a) \in \varphi(a)I(a)$ and consequently, $\frac{\psi(a)}{\varphi(a)} \in I(a)$. Since $E$ is a strongly coherent set of $\IntR(E,D)$, we obtain $\frac{\psi}{\varphi} \in I^{\Sk}$. Therefore, $\psi \in \varphi I^{\Sk}$, so we have $(\varphi I)^{\Sk} = \varphi I^{\Sk}$ for a nonzero $\varphi \in K(x)$ and a nonzero finitely-generated integral ideal $I$ of $\IntR(E,D)$. 
	
	Let $\varphi \in \IntR(E,D)$ be a nonzero element and let $I$ be a nonzero integral ideal of $\IntR(E,D)$. Any finitely-generated ideal contained in $\varphi I$ is of the form $\varphi J$ for some finitely-generated ideal $J$ of $\IntR(E,D)$ contained in $I$. Therefore, 
	\[
	(\varphi I)^{\Sk_f} = \bigcup_{J \in \Lambda_I} (\varphi J)^{\Sk} = \bigcup_{J \in \Lambda_I} \varphi J^{\Sk} = \varphi I^{\Sk_f},
	\]
	where $\Lambda_I$ is the family of all nonzero finitely-generated ideals of $\IntR(E,D)$ contained in $I$. 
	
	Next, let $I$ and $J$ be nonzero integral ideals of $\IntR(E,D)$. We can verify that $I\subseteq I^{\Sk_f}$ and if $I \subseteq J$, then $I^{\Sk_f} \subseteq J^{\Sk_f}$. Lastly, $(I^{\Sk_f})^{\Sk_f} = I^{\Sk_f}$. Therefore, $\Sk_f$ is indeed a star operation on $\IntR(E,D)$. Moreover, due to the definition of $\Sk_f$, we have $(\Sk_f)_f = \Sk_f$ and thus $\Sk_f$ is finite type. 
\end{proof}

However, using the notion of a strongly coherent set, defined in \cite{CahenChabertSkolem}, we can determine that if $E$ is a strongly coherent set for $\IntR(E,D)$, then $\Sk$ is a star operation for $\IntR(E,D)$. 

\begin{definition}
	Let $D$ be a domain with field of fractions $K$. A subset $E$ of $K$ is a \textbf{strongly coherent set} for $\IntR(E,D)$ if for every finitely-generated ideal $I$ of $\IntR(E,D)$ and every rational function $\varphi \in K(x)$ such that $\varphi(a) \in I(a)$ for all but finitely many $a \in E$, we have $\varphi \in I^{\Sk}$. 
\end{definition}

\begin{proposition}
	Let $D$ be a domain with field of fractions $K$. Let $E \subseteq K$ such that $E$ is a strongly coherent set for $\IntR(E,D)$. Then $\Sk$ is a star operation on $\IntR(E,D)$. 
\end{proposition}

\begin{proof}
	It suffices to check the properties of a star operation on integral ideals \cite[p. 393]{Gilmer}. We begin by letting $\varphi \in K(x)$ be a nonzero element and $I$ be a nonzero integral ideal of $\IntR(E,D)$ and showing that $(\varphi I)^{\Sk} = \varphi I^{\Sk}$. We immediately have $\varphi I^{\Sk} \subseteq (\varphi I)^{\Sk}$. Now let $\psi \in (\varphi I)^{\Sk}$. Then for all $a \in E$ except the finitely many zeroes of $\varphi$, we have $\psi(a) \in \varphi(a)I(a)$ and consequently, $\frac{\psi(a)}{\varphi(a)} \in I(a)$. Since $E$ is a strongly coherent set of $\IntR(E,D)$, we obtain $\frac{\psi}{\varphi} \in I^{\Sk}$. Therefore, $\psi \in \varphi I^{\Sk}$, so we have $(\varphi I)^{\Sk} = \varphi I^{\Sk}$ for a nonzero $\varphi \in K(x)$ and a nonzero integral ideal $I$ of $\IntR(E,D)$. 
	
	Finally, for any nonzero integral ideals $I, J$ of $\IntR(E,D)$ we have $I \subseteq J$ implying $I^{\Sk} \subseteq J^{\Sk}$ and $(I^{\Sk})^{\Sk} = I^{\Sk}$. Thus, ${}^{\Sk}$ is a star operation. 
\end{proof}

A use of star operations is to get a finer look at the ideals of a ring. We can get a finer look at the prime ideals of a ring as well. 

\begin{definition}
	Let $D$ be a domain with a star operation $\star$. If $I$ is a $\star$-ideal and a prime ideal of $D$, then we say that $I$ is a \textbf{$\star$-prime}. If $I$ is maximal among all $\star$-ideals with respect to inclusion, then we say that $I$ is a \textbf{$\star$-maximal} ideal. We denote by $\star-\Max(D)$ the set of all $\star$-maximal ideals of $D$. 
\end{definition}

\begin{remark}\cite{GabelliRoitman, Jaffard}
	A $\star$-maximal ideal, if it exists, is a $\star$-prime ideal. If $\star$ is of finite type, then $\star-\Max(D)$ is not empty. 
\end{remark}

Lastly, we also need the tools of the minimum valuation functions and local polynomials from \cite{Liu}. This allows us to consider a lot of the monomial valuations on a field of rational functions at once.

\begin{definition}
	Let $V$ be a valuation domain with value group $\Gamma$, valuation $v$, and field of fractions $K$. Take a nonzero polynomial $f \in K[x]$ and write it as $f(x) = a_nx^n + \cdots + a_1x+ a_0$ for $a_0, a_1, \dots, a_n \in K$. We define the \textbf{minimum valuation function of $f$} as $\minval_{f,v}: \Gamma \to \Gamma $ by \[\gamma \mapsto \min\{v(a_0), v(a_1)+\gamma, v(a_2)+2\gamma, \dots, v(a_n) + n\gamma\}\] for each $\gamma \in \Gamma$. We will denote $\minval_{f,v}$ as $\minval_f$ if the valuation $v$ is clear from context. Let $\Q\Gamma \coloneqq \Gamma \otimes_\Z \Q$. It is oftentimes helpful to think of $\minval_f$ as a function from $\Q{\Gamma}$ to $\Q{\Gamma}$ defined $\gamma \mapsto \min\{v(a_0), v(a_1)+\gamma, v(a_2)+2\gamma, \dots, v(a_n) + n\gamma\}$ for each $\gamma \in \Q{\Gamma}$. For a nonzero rational function $\varphi \in K(x)$, we may write $\varphi = \frac{f}{g}$ for some $f, g \in K[x]$ and define $\minval_\varphi = \minval_f - \minval_g$.

	In the same setup, taking $t \in K$, we can define the \textbf{local polynomial of $f$ at $t$} to be \[\loc_{f, v, t}(x) = \frac{f(tx)}{a_dt^d} \mod \m,\] where $\m$ is the maximal ideal of $V$ and $d = \max\{i \in \{0, 1, \dots, n\} \mid v(a_i) + iv(t) = \minval_f(v(t)) \}$. Again, we may omit the valuation $v$ in $\loc_{f, v, t}(x)$ and write $\loc_{f, t}(x)$ if the valuation is clear from the context.
\end{definition}

\section{The Skolem property and the maximal spectrum}\label{Sect:SkolemSpectrum}

\indent\indent In this section, we establish a connection between the Skolem property and the description of the maximal spectrum through ultrafilters. We then generalize a result about when a domain has a maximal spectrum that admits such a description. This will determine that the Skolem property holds for a family of rings of integer-valued rational functions larger than the family of Prüfer domains of integer-valued rational functions.

\begin{theorem}\label{Thm:UltraSkolem}
	Let $D$ be a domain with field of fractions $K$ and $E$ be a nonempty subset of $K$. Let $\Lambda$ be a dense subset of $\Max(D)$ with respect to the ultrafilter topology. Also define the family $\Pi = \{\M_{\m,a} \mid \m \in \Lambda, a \in E \}$. For each $\varphi \in \IntR(E,D)$, we define $\chi_\varphi = \{\M_{\m,a} \in \Pi \mid \varphi \in \M_{\m,a} \}$, and for an ideal $I \subseteq \IntR(E,D)$, we define $\chi_I = \{\chi_\varphi \mid \varphi \in I \}$. Then the following are equivalent:
	\begin{enumerate}
		\item $\chi_I$ has the finite intersection property for all proper ideals $I \subsetneq \IntR(E,D)$.
		\item  $\chi_I$ has the finite intersection property for all finitely-generated proper ideals $I \subsetneq \IntR(E,D)$. 
		\item  $\Pi$ is a dense subset of $\Max(\IntR(E,D))$ with respect to the ultrafilter topology.
		\item  $\IntR(E,D)$ has the Skolem property.
	\end{enumerate}
	
\end{theorem} 

This is a special case of Theorem \ref{Thm:StarUltraSkolem} with $\star_1 = d$ and $\star_2 = d$. We will prove the general case in the next section.

We utilize Theorem \ref{Thm:UltraSkolem} to establish the Skolem property for a domain of the form $\IntR(E,D)$ by determining if all of the maximal ideals of $\IntR(E,D)$ are ultrafilter limits of some family of pointed maximal ideals. The following proposition offers an instance of when the maximal ideals admit this description.

\begin{proposition}\cite[Proposition 2.8]{PvMD}\label{Prop:tUltra}
	Let $D$ be the intersection $D = \bigcap\limits_{\lambda \in \Lambda} V_\lambda$ of a family of valuation domains. For each $\lambda \in \Lambda$, denote by $\p_\lambda$, the center in $D$ of the maximal ideal of $V_\lambda$. 
	\begin{enumerate}
		\item If $I$ is a $t$-ideal of $D$, then $I$ is contained in the limit $\lim\limits_{\mathcal{F}} \p_\lambda$, of the family $\{\p_\lambda\}$, with respect to some filter $\mathcal{F}$. 
		\item If moreover $I$ is maximal, or if $I$ is $t$-maximal and every $V_\lambda$ is essential, i.e. $V_\lambda = D_{\p_\lambda}$, then $I = \lim\limits_{\mathcal{F}} \p_\lambda$. 
	\end{enumerate}
\end{proposition}

\begin{remark}
	If $D$ is a domain and $E$ is a subset of the field of fractions such that $\IntR(E,D)$ is a Prüfer domain, then the previous proposition implies that $\IntR(E,D)$ has the Skolem property. We write
	\[
	\IntR(E,D) = \bigcap_{\m \in \Max(D), a \in E} \IntR(E,D)_{\M_{\m,a}}
	\]
	with each $\IntR(E,D)_{\M_{\m,a}}$ being a valuation domain. Furthermore, every maximal ideal of $\IntR(E,D)$ is a $t$-ideal and therefore an ultrafilter limit of $\{\M_{\m,a} \mid \m \in \Max(D), a \in E \}$ by Proposition \ref{Prop:tUltra}(2). Then Theorem \ref{Thm:UltraSkolem} implies that $\IntR(E,D)$ has the Skolem property. This is essentially the same proof as \cite[Lemma 4.1]{IntValuedRational} but in the language of ultrafilters. 
\end{remark}

We now generalize Proposition \ref{Prop:tUltra} in order to show the Skolem property for a larger family of rings of integer-valued rational functions. 

\begin{proposition}\label{Prop:LocalIntersection}
	Let $D$ be the intersection $D = \bigcap\limits_\lambda D_\lambda$ of a family of local overrings. For each $\lambda$, denote by $\p_\lambda$ the center in $D$ of the maximal ideal $\m_\lambda$ of $D_\lambda$. 
	\begin{enumerate}
		\item  If $I$ is a $t$-ideal of $D$, then $I$ is contained in $\lim\limits_{\mathcal{F}} \p_\lambda$ with respect to some filter $\mathcal{F}$ of $\{\p_\lambda\}$. 
		\item  If moreover $I$ is maximal, or if $I$ is $t$-maximal and every $\p_\lambda$ is a $t$-ideal, then $I = \lim\limits_{\mathcal{F}} \p_\lambda$.  
	\end{enumerate}
\end{proposition}

\begin{proof} Let $I$ be a $t$-ideal of $D$ and take $J \subseteq I$ to be a finitely-generated ideal. We claim that $J \subseteq \p_\lambda$ for some $\lambda$. If not, then for each $\lambda$, there exists some $a \in J$ such that $a \notin \p_\lambda$. We can more specifically say that $a \in D_\lambda \setminus \m_\lambda$, so $a$ is a unit of $D_\lambda$. Now let $b \in J^{-1}$. Then $bJ \subseteq D$. In particular, we have $ba \in D \subseteq D_\lambda$. This implies that $b = baa^{-1} \in D_\lambda$. Thus, $J^{-1} \subseteq D_\lambda$ for each $\lambda$, which means $J^{-1} = D$. This is a contradiction, as this implies $J_v = D$, but $J_v \subseteq I \subsetneq D$. We see now that $\chi_{I} \coloneqq \{\chi_d \mid d \in I\}$ has the finite intersection property. This is because for $d_1, \dots, d_n \in I$, we have $(d_1, \dots, d_n) \subseteq \p_\lambda$ for some $\lambda$ and thus $\p_\lambda \in \chi_{d_1} \cap \cdots \cap \chi_{d_n}$. Then $\chi_I$ can be extended to a filter $\mathcal{F}$ of $\{\p_\lambda\}$, and this means that $I \subseteq \lim\limits_{\mathcal{F}} \p_\lambda$. 
	
	If we furthermore assume that $I$ is a maximal ideal, then $I = \lim\limits_{\mathcal{F}} \p_\lambda$. If we assume that $I$ is $t$-maximal and every $\p_\lambda$ is a $t$-ideal, then $\lim\limits_{\mathcal{F}} \p_\lambda$ is a $t$-ideal, so by the $t$-maximality of $I$, we must have $I = \lim\limits_{\mathcal{F}} \p_\lambda$.
\end{proof}

We can always write a domain $D$ as $D = \bigcap\limits_{\m \in \Lambda} D_\m$, where $\Lambda$ is some subset of $\Max(D)$. If all of the maximal ideals of $D$ are $t$-ideals, the previous proposition applies and shows that $\Max(D)$ is contained in the ultrafilter closure of $\Lambda$. A ring with such a property is a DW-domain.

\begin{definition}
	A domain in which $d=w$ as star operations is called a \textbf{DW-domain}. Equivalent for a domain $D$, every maximal ideal of $D$ is a $t$-ideal if and only if $D$ is a DW-domain \cite[Proposition 2.2]{Mimouni2}.
\end{definition}

We can then use the form of the maximal ideals of a DW-domain to show the Skolem property. 
\begin{theorem}
	Let $D$ be a domain with field of fractions $K$ and $E$ a subset of $K$. Suppose that $\IntR(E,D)$ is a DW-domain. Then $\IntR(E,D)$ has the Skolem property.
	
\end{theorem} 

\begin{proof}
	Denote by $D_{\m, a} = \{ \varphi \in K(x) \mid \varphi(a) \in D_{\m} \}$ for each $\m \in \Max(D)$ and $a \in E$. We have that \[\IntR(E,D) = \bigcap_{\m \in \Max(D), a \in E} D_{\m,a}.\] Notice that each $D_{\m,a}$ is a local domain centered on $\M_{\m,a}$. Because every maximal ideal of $\IntR(E,D)$ is a $t$-ideal, we use Proposition \ref{Prop:LocalIntersection} to conclude that all of the maximal ideals of $\IntR(E,D)$ are ultrafilter limits of $\{\M_{\m,a}\mid \m \in \Max(D), a \in E \}$. We satisfy Theorem \ref{Thm:UltraSkolem}(3) and thus $\IntR(E,D)$ has the Skolem property.
	
\end{proof}

If $D$ is a Prüfer domain, then $d = t$ since every nonzero finitely-generated ideal of $D$ is invertible. Therefore, $d = w$ and thus every Prüfer domain is a DW-domain. The previous result is therefore more general than the previously known result of $\IntR(E,D)$ being Prüfer implying the Skolem property \cite[Lemma 4.1]{IntValuedRational}. 

\section{The $(\star_1, \star_2)$-Skolem property}\label{Sect:StarStar}

\indent\indent In Theorem \ref{Thm:UltraSkolem}, the Skolem property is linked to a description of the maximal spectrum through ultrafilters. We notice that for a star operation $\star$ on a domain $D$, the $\star$-maximal ideals of $D$ are prime ideals of $D$. By Proposition \ref{Prop:FiltersPreserveStar}, if $\star$ is also a finite type star operation, then ultrafilter limits of a set of $\star$-prime ideals are also $\star$-prime, as long as the limit is not the zero ideal. This leads to a generalization of Theorem \ref{Thm:UltraSkolem} using star operations that gives rise to a more generalized notion of the Skolem property using star operations. 

\begin{definition}
	Let $D$ be a domain and $E$ a subset of the field of fractions of $D$. Also let $\star_1$ and $\star_2$ be star operations on $D$ and $\IntR(E,D)$, respectively. We say that $\IntR(E,D)$ has the \textbf{$(\star_1,\star_2)$-Skolem property} if $I \subseteq \IntR(E,D)$ is a finitely-generated ideal such that $I(a)^{\star_1} = D$ for each $a \in E$, then $I^{\star_2} = \IntR(E,D)$. Note that the $(d,d)$-Skolem property and the Skolem property are equivalent.
	
	The $(\star_1, \star_2)$-Skolem property is equivalent to the property that if $\varphi_1, \dots, \varphi_n$ are elements of $\IntR(E,D)$ such that $(\varphi_1, \dots, \varphi_n)^{\star_2}$ is a proper ideal of $\IntR(E,D)$, then there exists some $a \in E$ such that $(\varphi_1(a), \dots, \varphi_n(a))^{\star_1}$ is a proper ideal of $D$. 
\end{definition}

\begin{remark}
	Let $D$ be a domain and $E$ some subset of the field of fractions. If $I$ be a finitely-generated ideal of $\IntR(E,D)$ such that $I^{\star_2} = \IntR(E,D)$, it is not necessarily true that $I(a)^{\star_1} = D$ for all $a \in E$. 
	
	For example, if $D = \C[s]$, then $\IntR(D) = \C[x,s]$ and $(x,s) \subseteq \IntR(D)$ is a finitely-generated ideal such that $(x,s)_t = \IntR(D)$, yet $(x,s)(a)_t = (s)_t = (s)$ for $a \in s\C[s]$. 
\end{remark}

If we vary the star operations on $D$ and $\IntR(E,D)$, we can compare the star operation Skolem properties provided that the star operations are comparable in a certain way. 
\begin{proposition}
	Let $D$ be a domain with field of fractions $K$ and $E$ be a subset of $K$. Suppose that $\star_1, \star_1'$ are two star operations on $D$ such that $\star_1' \leq \star_1$, and $\star_2, \star_2'$ are two star operations on $\IntR(E,D)$ such that $\star_2 \leq \star_2'$. If $\IntR(E,D)$ has the $(\star_1, \star_2)$-Skolem property, then $\IntR(E,D)$ has the $(\star_1',\star_2')$-Skolem property. 
\end{proposition}

\begin{proof} Let $I$ be a nonzero finitely-generated ideal of $\IntR(E,D)$ such that $I^{\star_2'}$ is proper. Since $I^{\star_2} \subseteq I^{\star_2'} \subsetneq \IntR(E,D)$ and $\IntR(E,D)$ has the $(\star_1,\star_2)$-Skolem property, there exists $a \in E$ such that $I(a)^{\star_1}$ is a proper ideal of $D$. By assumption, $I(a)^{\star_1'} \subseteq I(a)^{\star_1}$, so we have that $I(a)^{\star_1'}$ is proper. Thus, $\IntR(E,D)$ has the $(\star_1',\star_2')$-Skolem property.
\end{proof}

We now give the analogue of Theorem \ref{Thm:UltraSkolem} in terms of the star operations. 
\begin{theorem}\label{Thm:StarUltraSkolem}
	Let $D$ be a domain with field of fractions $K$ and $E$ be a nonempty subset of $K$. Take $\star_1$ to be a finite type star operation on $D$ and $\star_2$ to be a finite type star operation on $\IntR(E,D)$. Let $\Lambda \subseteq \Spec(D)$ such that each $\p \in \Lambda$ has the property that $\p^{\star_1} \subsetneq D$ and $\star_1-\Max(D)$ is contained in the ultrafilter closure of $\Lambda$. Define the family $\Pi = \{\P_{\p,a} \mid \p \in \Lambda, a \in E \}$. We assume every ideal in $\Pi$ is a $\star_2$-prime ideal. 
	
	For an element $\varphi \in \IntR(E,D)$, we define $\chi_\varphi = \{\P_{\p,a} \in \Pi \mid \varphi \in \P_{\p,a} \}$, and for an ideal $I \subseteq \IntR(E,D)$, we define $\chi_I = \{\chi_\varphi \mid \varphi \in I \}$. Then the following are equivalent:
	\begin{enumerate}
		\item  $\chi_I$ has the finite intersection property for all ideals $I \subseteq \IntR(E,D)$ such that $I^{\star_2} \subsetneq \IntR(E,D)$.
		\item  $\chi_I$ has the finite intersection property for every finitely-generated ideals $I$ of $\IntR(E,D)$ such that $I^{\star_2} \subsetneq \IntR(E,D)$.
		\item  $\star_2-\Max(\IntR(E,D))$ is contained in the ultrafilter closure of $\Pi$.
		\item  $\IntR(E,D)$ has the $(\star_1, \star_2)$-Skolem property.
	\end{enumerate}
\end{theorem}

\begin{proof}
	Again, we first establish the equivalence of 1, 2, and 3. Then, we show that 2 and 4 are equivalent.

	\begin{itemize}
		\setlength{\itemindent}{3em}
		\item[$1 \implies 2:$] Finitely-generated ideals are ideals.
		
		\item[$2 \implies 1:$] Let $I \subseteq \IntR(E,D)$ such that $I^{\star_2}$ is proper. Take $\varphi_1, \dots, \varphi_n \in I$. Then $(\varphi_1, \dots, \varphi_n)^{\star_2} \subseteq I^{\star_2}$ is proper, so $\chi_{(\varphi_1, \dots, \varphi_n)}$ has the finite intersection property. In particular, $\chi_{\varphi_1} \cap \cdots \cap \chi_{\varphi_n}$ is nonempty. 
		
		\item[$1 \implies 3:$] Let $\P$ be a $\star_2$-maximal ideal of $\IntR(E,D)$. Then $\chi_\P$ has the finite intersection property so there exists some ultrafilter $\mathcal{U}$ of $\Pi$ such that $\chi_\P \subseteq \mathcal{U}$. This implies that $\P \subseteq \lim\limits_{\mathcal{U}} \P_{\p, a}$, the ultrafilter limit of $\Pi$ with respect to $\mathcal{U}$. Since each $\P_{\p,a} \in \Pi$ is a $\star_2$-ideal, by Proposition \ref{Prop:FiltersPreserveStar}, we have that $\lim\limits_{\mathcal{U}} \P_{\p, a}$ is a $\star_2$-ideal. Together with the fact that $\P$ is $\star_2$-maximal, we conclude that $\P = \lim\limits_{\mathcal{U}} \P_{\p, a}$. 
		
		\item[$3 \implies 1:$] Let $I \subseteq \IntR(E,D)$ such that $I^{\star_2}$ is a proper ideal of $\IntR(E,D)$. This means that $I \subseteq \lim\limits_{\mathcal{U}} \P_{\p, a}$, the ultrafilter limit of $\Pi$ with respect to $\mathcal{U}$, for some ultrafilter $\mathcal{U}$ of $\Pi$. This shows that $\chi_I \subseteq \mathcal{U}$ and therefore $\chi_I$ has the finite intersection property. 
		
		\item[$2 \implies 4:$] Let $(\varphi_1, \dots, \varphi_n) \subseteq \IntR(E,D)$ such that $(\varphi_1, \dots, \varphi_n)^{\star_2}$ is a proper ideal. Then $\chi_{\varphi_1} \cap \cdots \cap \chi_{\varphi_n}$ is not empty so it contains some $\P_{\p,a} \in \Pi$. This implies that $\varphi_1(a), \dots, \varphi_n(a) \in \p$. Thus, $(\varphi_1(a), \dots, \varphi_n(a))^{\star_1} \subseteq \p^{\star_1}$, a proper ideal of $D$. 
		
		\item[$4 \implies 2:$] Let $(\varphi_1, \dots, \varphi_n) \subseteq \IntR(E,D)$ such that $(\varphi_1, \dots, \varphi_n)^{\star_2}$ is a proper ideal of $\IntR(E,D)$. Due to the fact that $\IntR(E,D)$ has the $(\star_1, \star_2)$-Skolem property, there exists $a \in E$ such that $(\varphi_1(a), \dots, \varphi_n(a))^{\star_1}$ is proper. This means that $(\varphi_1(a), \dots, \varphi_n(a)) \subseteq \q$ for some $\q \in \star_1-\Max(D)$. Since the ultrafilter closure of $\Lambda$ contains all $\star_1$-maximal ideals and $(\varphi_1(a), \dots, \varphi_n(a))$ is finitely-generated, there exists $\p \in \Lambda$ so that $(\varphi_1(a), \dots, \varphi_n(a)) \subseteq \p$. Thus, $\P_{\p, a} \in \chi_{\varphi_1} \cap \cdots \cap \chi_{\varphi_n}$. We conclude that $\chi_I$ has the finite intersection property. 
	\end{itemize}

\end{proof}

The extra assumption that every ideal in $\Pi$ is a $\star_2$-prime ideal is achieved if $\star_1$ and $\star_2$ are integrally compatible.

\begin{definition}
	Let $D$ be a domain and take $E$ to be a subset of the field of fractions of $D$. Let $\star_1$ be a star operation on $D$ and $\star_2$ be a star operation on $\IntR(E,D)$. We say that $\star_1$ and $\star_2$ are \textbf{integrally compatible} if for all ideals $I$ of $\IntR(E,D)$, we have $I^{\star_2}(a)^{\star_1} = I(a)^{\star_1}$ for all $a \in E$.
\end{definition}

Compare this to the definition of compatibility of star operations on rings in a ring extension given in \cite{Anderson}. For domains $A$ and $B$ with $A \subseteq B$ and star operations $\star_1$ on $A$ and $\star_2$ on $B$, we say that $\star_1$ and $\star_2$ are compatible if for any fractional ideal $I$ of $A$, we have $(IB)^{\star_2} = (I^\star_1B)^{\star_2}$. The comparison of ideals for compatibility happens in the larger ring, whereas the comparison of ideals for integral compatibility happens in the smaller ring.  

\begin{proposition}
	Let $D$ be a domain and take $E$ to be a subset of the field of fractions of $D$. Also let $\star_1$ be a star operation on $D$ and $\star_2$ be a star operation on $\IntR(E,D)$. Suppose that $\star_1$ and $\star_2$ are integrally compatible. If $\p$ is a $\star_1$-prime of $D$, then $\P_{\p,a}$ is a $\star_2$-prime of $\IntR(E,D)$ for any $a \in E$. 
\end{proposition}

\begin{proof}
	Let $\p$ be a $\star_1$-prime of $D$ and let $a \in E$. We want to show that $\P_{\p, a}^{\star_2} \subseteq \P_{\p, a}$. To this end, we calculate that
	\[
	\P_{\p,a}^{\star_2}(a)^{\star_1} = \P_{\p,a}(a)^{\star_1} = \p^{\star_1} = \p.
	\]
	Thus, $\P_{\p,a}^{\star_2}(a) \subseteq \p$, which means that $\P_{\p, a}^{\star_2} \subseteq \P_{\p, a}$. This implies that $\P_{\p, a}$ is a $\star_2$-prime of $\IntR(E,D)$.
\end{proof}

We want to extend Proposition \ref{Prop:tUltra} using star operations in order to use Theorem \ref{Thm:StarUltraSkolem} to show the $(\star_1,\star_2)$-Skolem property. The following results discuss star operations that are $w$-operations. It is another generalization of Proposition \ref{Prop:tUltra}. Instead of generalizing the family of ring of which $D$ can be written as the intersection, we generalize the star operation. 

\begin{proposition}\label{Prop:StarUltrafilter}
	Let $D = \bigcap\limits_\lambda V_\lambda$ be an intersection of a family of valuation domains. Let $\star$ denote the star operation on $D$ given by $I^\star = \bigcap\limits_\lambda IV_\lambda$ for every nonzero fractional ideal of $D$, and denote by $\p_\lambda$ the center of $V_\lambda$ in $D$. Then the following statements hold.
	\begin{enumerate}
		\item If $I\subsetneq D$ is a $\star_f$-ideal of $D$, then $I \subseteq \lim\limits_{\mathcal{F}} \p_\lambda$ for some filter $\mathcal{F}$ of $\{\p_\lambda\}$.
		\item If moreover $I$ is maximal, or if $I$ is $\star_f$-maximal and every $\p_\lambda$ is a $\star_f$-ideal, then $I = \lim\limits_{\mathcal{F}} \p_\lambda$. 
	\end{enumerate}
\end{proposition}
\begin{proof}
	\,
	\begin{enumerate}
		\item Let $J \subseteq I$ be a finitely-generated ideal. If $J \not\subseteq \p_\lambda$, then $JV_\lambda = V_\lambda$. This is because $JV_\lambda \subseteq \m_\lambda$, the maximal ideal of $V_\lambda$, implies $J \subseteq JV_\lambda \cap D \subseteq \p_\lambda$. Thus, if $J \not\subseteq \p_\lambda$ for all $\lambda$, then $J^\star = D$. This is a contradiction since $D = J^\star \subseteq I^{\star_f}$, making $I$ not a $\star_f$-ideal.  
		
		This implies for all finitely-generated ideals $J \subseteq I$, there exists some $\lambda$ such that $J \subseteq \p_\lambda$. In particular, take $a_1, \dots, a_n \in I$. Then there exists a $\lambda$ such that $(a_1, \dots, a_n) \subseteq \p_\lambda$. This means that $\chi_{a_1} \cap \cdots \cap \chi_{a_n}$ is nonempty, so the collection $\chi_I \coloneqq \{\chi_d \mid d \in I\}$ can be extended to some filter $\mathcal{F}$ of $\{\p_\lambda\}$. This implies that $I \subseteq \lim\limits_{\mathcal{F}} \p_\lambda$.

		\item  If moreover $I$ is a maximal ideal, then $I \subseteq \lim\limits_{\mathcal{F}} \p_\lambda$ implies $I = \lim\limits_{\mathcal{F}} \p_\lambda$. 
		
		If each $\p_\lambda$ is a $\star_f$-ideal, then $\lim\limits_{\mathcal{F}} \p_\lambda$ is a $\star_f$-ideal by Proposition \ref{Prop:FiltersPreserveStar} since $\lim\limits_{\mathcal{F}} \p_\lambda$ is nonzero. Therefore, if $I$ is $\star_f$-maximal, then $I = \lim\limits_{\mathcal{F}} \p_\lambda$. 
	\end{enumerate}
\end{proof} 

\begin{corollary}\label{Cor:StarUltrafilter}
	Let $D = \bigcap\limits_\lambda V_\lambda$ be an intersection of a family of essential valuation domains, that is, each $V_\lambda$ is essential for $D$. Let $\star$ denote the star operation on $D$ given by $I^\star = \bigcap\limits_\lambda IV_\lambda$ for every nonzero fractional ideal of $D$, and denote by $\p_\lambda$ the center of $V_\lambda$ in $D$. If $I$ is a $\star_f$-maximal ideal of $D$, then $I = \lim\limits_{\mathcal{U}} \p_\lambda$ for some ultrafilter $\mathcal{U}$ of $\{\p_\lambda\}$. 
	
\end{corollary}

\begin{proof}
	The maximal ideal of a valuation domain is always a $t$-ideal, so by \cite[Lemma 3.17(1)]{Kang}, we know that each $\p_\lambda$ is also a $t$-ideal. Because $\star_f \leq t$, we have that each $\p_\lambda$ is a $\star_f$-ideal. Thus, if $I$ is a $\star_f$-maximal ideal of $D$, then $I = \lim\limits_{\mathcal{U}} \p_\lambda$ for some ultrafilter $\mathcal{U}$ of $\{\p_\lambda\}$ by Proposition \ref{Prop:StarUltrafilter}. 
\end{proof}

We then use this result to show that for certain star operations $\star_1$ on $D$ and $\star_2$ on $\IntR(E,D)$ built using valuation overrings, we have the $(\star_1, \star_2)$-Skolem property for $\IntR(E,D)$.

\begin{theorem}
	Let $D$ be a domain with field of fractions $K$. Take $E$ to be some subset of $K$. Let $\star_1$ be a finite type star operation on $D$ and $\Lambda \subseteq \Spec(D)$ such that
	\begin{itemize}
		\item $\p \in \Lambda$ implies that $\p^{\star_1}$ is a proper ideal,
		\item  $D = \bigcap\limits_{\p \in \Lambda} D_\p$, and
		\item $\star_1-\Max(D)$ is contained in the ultrafilter closure of $\Lambda$.
	\end{itemize}
	Let $\star_2$ be a star operation on $\IntR(E,D)$ defined by $I^{\star_2} = \bigcap\limits_{\mu} IV_\mu$ for every nonzero fractional ideal $I$ of $D$, where $\{V_\mu\}$ is a family of valuation overrings such that
	\begin{itemize}
		\item $\IntR(E,D) = \bigcap\limits_\mu V_\mu$,
		\item each $V_\mu$ is centered on $\P_{\p, a}$ for some $\p \in \Lambda$ and $a \in E$, and
		\item for each $\p \in \Lambda$ and $a \in E$, the ideal $\P_{\p, a}$ is a $(\star_2)_f$ ideal. 
	\end{itemize}
	Then $\IntR(E,D)$ has the $(\star_1, (\star_2)_f)$-Skolem property. 
\end{theorem}

\begin{proof}
	Let $I$ be a $(\star_2)_f$-maximal ideal of $\IntR(E,D)$. Then by Corollary \ref{Cor:StarUltrafilter}, we have that $I$ is an ultrafilter limit of $\{\M_\mu \cap \IntR(E,D) \}$, where $\M_\mu$ is the maximal ideal of $V_\mu$. Since $\{\M_\mu \cap \IntR(E,D) \} \subseteq \{\P_{\p, a} \mid \p \in \Lambda, a \in E \}$, we know that $I$ is also an ultrafilter limit of $\{\P_{\p, a} \mid \p \in \Lambda, a \in E \}$. By Theorem \ref{Thm:StarUltraSkolem}, we get that $\IntR(E,D)$ has the $(\star_1, (\star_2)_f)$-Skolem property.
\end{proof}

\section{Failure of the strong Skolem property}\label{Sect:StrongSkolem}
\indent\indent In this section, we give examples of rings of integer-valued rational functions which do not have the strong Skolem property. First, we show that if $D$ is a pseudovaluation domain whose associated valuation domain has a principal maximal ideal, then $\IntR(D)$ has the Skolem property but not the strong Skolem property. Then we analyze the case when $\IntR(V)$ is not Prüfer, for $V$ a valuation domain. In this case, $\IntR(V)$ does not have the strong Skolem property and we give a description of the finitely-generated ideals of $\IntR(V)$ that are not Skolem closed.

If $D$ is a domain such that $\IntR(D)$ is Prüfer and has the Skolem property, then $\IntR(D)$ automatically has the strong Skolem property. However, it is not necessarily the case that the Skolem property implies the strong Skolem property if $\IntR(D)$ is not Prüfer. We can being constructing such an example using a PVD $D$ whose associated valuation domain has principal maximal ideal. We show that with some conditions on residue fields, we can find an ideal that is not Skolem closed in $\IntR(D)$. Further assumptions will make this ideal finitely-generated, showing that under some conditions, $\IntR(D)$ can have the Skolem property but not the strong Skolem property.

First, we define a PVD. These are local rings that behave almost like valuation domains. For references on these rings, see \cite{HedstromHouston, HedstromHoustonII}. 
\begin{definition}
	A domain $D$ is a \textbf{pseudovaluation domain} (PVD) if $D$ has a valuation overring $V$ such that $\Spec(D) = \Spec(V)$ as sets. The valuation domain is uniquely determined and is called the \textbf{associated valuation domain} of $D$.
\end{definition}

\begin{remark}
	In particular, a pseudovaluation domain and the associated valuation domain have the same (unique) maximal ideal. 
	
	One way to construct a pseudovaluation domain is to start with a valuation domain $V$. Let $\m$ be the maximal ideal of $V$. Consider the canonical projection $\pi:V \to V/\m$. Take a subfield $F \subseteq V/\m$. Then $D \coloneqq \pi^{-1}(F)$ is a pseudovaluation domain with associated valuation domain $V$. 
\end{remark}

We will focus on PVDs $D$ whose associated valuation domain has a principal maximal ideal. In this case, we can make use of Theorem \ref{Thm:UltraSkolem} to show that $\IntR(E,D)$ has the Skolem property for any nonempty subset $E$ of the field of fractions of $D$. We do this by showing that all of the maximal ideals of $\IntR(E,D)$ are ultrafilter limits of maximal pointed ideals. This proof makes use of the rational function $\theta$ found in \cite[Theorem 3.5]{IntValuedRational} setting $n = 2$. 

\begin{proposition}\label{Prop:PseudosingularPVDSkolem}
	Let $D$ be a PVD whose associated valuation domain $V$ has a principal maximal ideal. Also let $E$ be a nonempty subset of the field of fractions $K$ of $D$. Then the ring $\IntR(E,D)$ has the Skolem property. 
\end{proposition}

\begin{proof}
For each $\varphi \in \IntR(E,D)$, we set $\chi_\varphi = \{\M_{\m, a} \mid \varphi \in \M_{\m, a}, a \in E \}$, where $\m$ is the maximal ideal of $D$. Also let $t \in \m$ be a generator of $\m$ in $V$.

Let $A \subseteq \IntR(E,D)$ be a proper ideal. We first want to show that $\chi_A \coloneqq \{\chi_\varphi \mid \varphi \in A \}$ is closed under finite intersections. Take $\varphi_1, \varphi_2 \in A$. If $\varphi_2 = 0$, then $\chi_{\varphi_2} = \{\M_{\m, a} \mid a \in E\}$, so $\chi_{\varphi_1} \cap \chi_{\varphi_2} = \chi_{\varphi_1}$. Now suppose that $\varphi_2 \neq 0$. Set \[\theta(x) = \frac{t(1+x^{4})}{(1+tx^2)(t+x^2)}.\] We claim that $\theta \in \IntR(K,D)$. Denote by $v$ the valuation associated with $V$. Take $a \in K$. If $v(a) = 0$, then $v(\theta(a)) = v(t) + v(1+a^{4}) - 0 - 0 > 0$, so $\theta(a) \in \m  \subseteq D$. If $v(a) > 0$, then 
\[
\theta(a) = \frac{t(1+a^{4})}{(1+ta^2)(t+a^2)} = \frac{1+a^{4}}{(1+ta^2)(1+ \frac{a^2}{t})}. 
\] 
We have $1+a^{4} \equiv 1 \pmod \m$ and $(1+ta^2)(1+ \frac{a^2}{t}) \equiv 1 \cdot 1 \equiv 1 \pmod \m$. This means that $\theta(a) \in 1+ \m \subseteq D$. Lastly, suppose that $v(a) < 0$. We calculate that
\[
\theta(a) = \frac{t(1+a^{4})}{(1+ta^2)(t+a^2)} = \frac{\frac{1}{a^{4}}+1}{(\frac{1}{ta^2}+1)(\frac{t}{a^2}+ 1)}. 
\]
We see that $\frac{1}{a^{4}}+1 \equiv 1 \pmod \m$ and also $(\frac{1}{ta^2}+1)(\frac{t}{a^2}+ 1) \equiv 1 \cdot 1 \equiv 1 \pmod \m$. Thus, $\theta(a) \in 1+ \m \subseteq D$. Since $\theta(a) \in D$ for all $a \in K$, we have that $\theta \in \IntR(K,D)$. 

Now we consider $\rho(x) = \varphi_1(x) + \theta\left( \frac{\varphi_1(x)}{\varphi_2(x)}\right) \varphi_2(x)$. We know that $\varphi_2 \neq 0$ so $\varphi_2$ has only finitely many roots. Thus, $\theta\left( \frac{\varphi_1(x)}{\varphi_2(x)}\right) \in \IntR(K,D)$ as this rational function maps almost every element of $K$ to $D$ \cite[Proposition 1.4]{Liu}. This implies that $\theta\left( \frac{\varphi_1(x)}{\varphi_2(x)}\right) \in \IntR(E,D)$ and therefore $\rho(x) \in (\varphi_1, \varphi_2) \subseteq A$. 

Let $a \in E$. Supposing that $v(\varphi_1(a)) = v(\varphi_2(a))$, we have that $v\left(\theta\left(\frac{\varphi_1(a)}{\varphi_2(a)}\right) \right) > 0$ since we have shown before that $\theta(b) \in \m$ for any $b \in K$ such that $v(b) = 0$. Then $v(\rho(a)) = v(\varphi_1(a))$. If $v(\varphi_1(a)) < v(\varphi_2(a)))$, we have $v(\rho(a)) = v(\varphi_1(a))$. If $v(\varphi_1(a)) > v(\varphi_2(a))$, we have $v(\rho(a)) = v(\varphi_2(a))$. In summary, $v(\rho(a)) = \min\{v(\varphi_1(a)), v(\varphi_2(a)) \}$. This implies that $\chi_{\varphi_1} \cap \chi_{\varphi_2} = \chi_{\rho}$. Thus, $\chi_A$ is closed under finite intersections.

Since $A$ is a proper ideal, $\chi_A$ does not contain the empty set. If $\chi_A$ contained the empty set, then there would be some $\varphi \in A$ such that $\chi_\varphi = \emptyset$, which means that $\varphi(a)$ is a unit for all $a \in E$, making $\varphi$ a unit of $\IntR(E,D)$ and thus $A$ not a proper ideal. We have also just shown that $\chi_A$ is closed under finite intersections, so $\chi_A$ has the finite intersection property. By Theorem \ref{Thm:UltraSkolem}, we know that $\IntR(E,D)$ has the Skolem property. 
\end{proof}

For a PVD $D$ whose associated valuation domain has a principal maximal ideal, even though $\IntR(D)$ has the Skolem property, we can find some ideals of $\IntR(D)$ that are not Skolem closed. 

\begin{proposition}\label{Prop:PVDExample}
	Let $D$ be a PVD whose associated valuation domain $V$ has principal maximal ideal $\m$. Assume that $V/\m$ is infinite and $V \neq D$. Then $(x^2, \m)$ of $\IntR(D)$ is not Skolem closed.
\end{proposition}

\begin{proof}
	We claim that $x$ is in the Skolem closure of $(x^2, \m)$. If $a \in D$ is such that $v(a) = 0$, then the value ideal of $(x^2, \m)$ at $a$ is $D$. If $a \in D$ is such that $v(a) > 0$, then the value ideal of $(x^2, \m)$ at $a$ is $\m$. Thus, $x \in (x^2 , \m )^{\Sk}$.
	
	However, we claim that $x \notin (x^2, \m)$. If $x \in (x^2, \m)$, then
	\[
	x = \varphi(x)x^2 + \sum_{i=1}^{n} \psi_i(x) t_i, 
	\]
	for some $t_1, \dots, t_n \in \m$ and $\varphi, \psi_1, \dots, \psi_n \in \IntR(D)$. Let $t \in V$ be such that $t$ generates $\m$ in $V$. We can permute the indices so that for some $m \in \N$, we have $v(t_1) = \cdots = v(t_m) = v(t)$ and $v(t) < v(t_i)$ for any $i > m$. We cannot have $m = 0$ because otherwise evaluation at $x = t$ yields $t = \varphi(t)t^2 + \sum\limits_{i=1}^{n}\psi_i(t)t_i$ and the $v$ valuation of the right hand side is strictly greater than $v(t)$. We now rewrite
	\[
	\psi_{m+1}(x) t_{m+1} + \cdots +\psi_n(x) t_n = t^2 \left(\psi_{m+1}(x) \frac{t_{m+1}}{t^2} + \cdots + \psi_n(x) \frac{t_n}{t^2} \right) = t^2\chi(x),
	\]
	where $\chi(x) = \psi_{m+1}(x) \frac{t_{m+1}}{t^2} + \cdots + \psi_n(x) \frac{t_n}{t^2} \in \IntR(D,V)$. This is because for each $i > m$, the fact that $v(t_i) > v(t)$ and that $\m$ is principal in $V$ implies $v(t_i) \geq v(t^2)$ and thus $\frac{t_i}{t^2} \in V$. Therefore, we now are considering
	\[
	x = \varphi(x)x^2 + \chi(x)t^2 + \sum_{i=1}^{m} \psi_i(x) t_i. 
	\]
	We now assume that $\frac{t_1}{t}, \dots, \frac{t_m}{t} \mod \m$ are linearly independent over $D/\m$. If this is not the case, then there is some relation $d_1\frac{t_1}{t} + \cdots + d_m\frac{t_m}{t} \in \m$ for some $d_1, \dots, d_m \in D$ and not all $d_1, \dots, d_m$ are in $\m$. Without loss of generality, we can assume that $d_1 \notin \m$. Then $t_1 = -\frac{d_2}{d_1}  t_2 - \cdots -\frac{d_m}{d_1}  t_m$. This allows us to write $\sum\limits_{i=1}^{m} \psi_i(x) t_i$ as an $\IntR(D)$-linear combination of just $t_2, \dots, t_m$. We may keep going until we have linear independence. 
	
	Now let $s \in V$ such that $v(s) = v(t)$. Then $s = \varphi(s)s^2 + \chi(s)t^2 + \sum\limits_{i=1}^{m} \psi_i(s) t_i$. Dividing both sides by $t$ and then taking both sides modulo $\m$, we have 
	\[
		\frac{s}{t} = \sum_{i=1}^{m} \psi_i(s) \frac{t_i}{t} \mod \m. 
	\]
	Since each nonzero element of $V/\m$ can be represented by $\frac{s}{t}$ for some $s \in V$ with $v(s) = v(t)$, this shows that $V/\m = D/\m(\frac{t_1}{t} + \m, \dots, \frac{t_m}{t} + \m)$. This also shows that $m > 1$ since otherwise, $D/\m = V/\m$. Moreover, $D/\m$ must also be infinite.
	
	Now let $a_1, \dots, a_m \in D$ not all in $\m$. Also let $s = a_1t_1 + \cdots + a_mt_m$. Since $\frac{t_1}{t}, \dots, \frac{t_m}{t} \mod \m$ are linearly independent over $D/\m$, we know that $v(s) = v(t)$. This implies that
	\[
	a_1\frac{t_1}{t} + \cdots + a_m \frac{t_m}{t} = \sum_{i=1}^m\psi_i(s)\frac{t_i}{t} \mod \m.
	\]
	Due to linear independence of $\frac{t_1}{t}, \dots, \frac{t_m}{t} \mod \m$, we have 
	\[
	a_i = \psi_i(a_1t_1 + \cdots + a_mt_m) \mod \m
	\]
	for each $i\in \{1,\dots, m\}$. 
	
	Fix an $i \in \{1, \dots, m\}$. Write $\psi_i = \frac{f_i}{g_i}$ for some $f_i, g_i \in K[x]$, where $K$ is the field of fractions of $D$. Let $u \in V$ with $v(u) = 0$ such that $u + \m$ is not a root of $\loc_{f_i, t}(x)$ or $\loc_{g_i, t}(x)$. Then by \cite[Proposition 2.19]{Liu}, we know that $v(f_i(tu)) = \minval_{f_i}(v(t))$ and $v(g_i(tu)) = \minval_{g_i}(v(t))$. Thus, $v(\psi(tu)) = \minval_{\psi_i}(v(t))$. We can write $u \in V$ with $v(u) = 0$ as $u = a_1 t_1 + \cdots + a_m t_m$ for some $a_1, \dots, a_m \in D$ not all in $\m$. Since $m > 1$ and $D/\m$ is infinite, we can choose $u \in V$ such that $u + \m$ is not root of $\loc_{f_i, t}(x)$ or $\loc_{g_i, t}(x)$ and $a_i \in \m$. This implies that $\psi_i(tu) \in \m$ and thus $\minval_{\psi_i}(v(t)) > 0$. We can likewise choose $u \in V$ such that $u + \m$ is not root of $\loc_{f_i, t}(x)$ or $\loc_{g_i, t}(x)$ and $a_i \notin \m$. This will imply that $\psi_i(tu) \notin \m$ and thus $\minval_{\psi_i}(v(t)) = 0$, a contradiction. Thus, $x \notin (x^2, \m)$, which means $(x^2, \m)$ is not Skolem closed. 
\end{proof}

If $\m$ is a finitely-generated ideal of $D$, then $(x^2,\m)$ is a finitely-generated ideal of $\IntR(D)$, meaning $\IntR(D)$ has a finitely-generated ideal that is not Skolem closed, so $\IntR(D)$ does not have the strong Skolem property. The next proposition shows that the condition that $\m$ is finitely-generated in $D$ can be easily met. 

\begin{proposition}\label{Prop:FinDegResFieldExt}
	Let $D$ be a PVD with maximal ideal $\m$ and associated valuation domain $V$. Then $\m$ is a finitely-generated ideal of $D$ if and only if $[V/\m: D/\m] < \infty$ and $\m$ is principal in $V$. 
\end{proposition}

\begin{proof}
	We know that $\m$ is a finitely-generated ideal of $D$ if and only if $V$ is a finitely-generated $D$-module and $\m$ is principal in $V$ \cite[Proposition 1.5]{HedstromHoustonII}. What is left to show is that under the assumption that $\m$ is principal in $V$, we have $[V/\m:D/\m] < \infty$ if and only if $V$ is a finitely-generated $D$-module.
	
	Suppose that $V = Dt_1 + \cdots + Dt_n$ for some $t_1, \dots, t_n \in V$. Then viewing this modulo $\m$, we get $V/\m = (D/\m)(t_1 + \m) + \cdots + (D/\m) (t_n + \m)$. This implies that $[V/\m : D/\m] < \infty$. 
	
	Now suppose that $[V/\m:D/\m] < \infty$. Then there exists some $t_1, \dots, t_n \in V^\times$ such that $V/\m = (D/\m)(t_1 + \m)+ \cdots + (D/\m)(t_n + \m)$. Since $\m$ is principal in $V$, we know $\m=tV$ for some $t \in V$. We claim that $\m = (tt_1, \dots, tt_n)D$. Because $t_1, \dots, t_n \in V$, we know that $tt_1, \dots, tt_n \in \m$. Therefore, $(tt_1, \dots, tt_n)D \subseteq \m$. Now take $d \in \m$. Then $\frac{d}{t} \in V$, so $\frac{d}{t} = a + c_1t_1 + \cdots + c_nt_n$ for some $c_1, \dots, c_n \in D$ and $a \in \m$. This implies that $d = at + c_1tt_1 + \cdots c_n tt_n$. Notice that $\frac{at}{tt_1} = \frac{a}{t_1} \in \m$ as $a \in \m$ and $t_1 \in V^\times$. Thus, $at \in tt_1D \subseteq (tt_1, \dots, tt_n)D$, meaning that $d \in (tt_1, \dots, tt_n)D$. This shows that $\m = (tt_1, \dots, tt_n)D$. Now, $V = Dt_1 + \cdots + Dt_n + \m = Dt_1 + \cdots + Dt_n$, so $V$ is a finitely-generated $D$-module. 
\end{proof} 

Note that $\IntR(D)$ in the following corollary has the Skolem property by Proposition \ref{Prop:PseudosingularPVDSkolem}, but $\IntR(D)$ does not have the strong Skolem property. 

\begin{corollary}
	Let $D$ be a PVD with associated valuation domain $V$ that has a principal maximal ideal $\m$. If $V/\m$ is infinite and $V \neq D$, then $\IntR(D)$ does not have the strong Skolem property. 
\end{corollary}

\begin{proof}
	By Proposition \ref{Prop:FinDegResFieldExt}, we know that $(x^2, \m)$ is a finitely-generated ideal of $\IntR(D)$. We also know that $(x^2, \m)$ is not Skolem closed by \ref{Prop:PVDExample}. Thus, $(x^2, \m)$ is a finitely-generated ideal of $\IntR(D)$ that is not Skolem closed, implying that $\IntR(D)$ does not have the Skolem property. 
\end{proof}

Let $V$ be a valuation domain with field of fractions $K$. Take $E$ to be a nonempty subset of $K$ that is strongly coherent for $\IntR(E,V)$. The following lemma shows that $\IntR(E,V)$ is integrally closed, so we have that $\IntR(E,V)$ is not Prüfer if and only if there exists a nonzero finitely-generated ideal of $\IntR(V)$ that is not divisorial \cite[Proposition 34.12]{Gilmer}. It turns out this is a perspective that links Prüfer domains and the strong Skolem property.

\begin{lemma}\label{Lem:IntegrallyClosedTransfers}
	Let $D$ be a domain with field of fractions $K$. Suppose that $E$ is a nonempty subset of $K$. Then $\IntR(E,D)$ is integrally closed if and only if $D$ is integrally closed.
\end{lemma}

\begin{proof}
	If $\IntR(E,D)$ is integrally closed, then $D = \IntR(E,D) \cap K$ is also integrally closed. 
	
	Now suppose that $D$ is integrally closed. Also suppose that $\varphi \in K(x)$ is integral over $\IntR(E,D)$. Then there exists $\psi_0, \dots, \psi_{n-1} \in \IntR(E,D)$ such that $$\varphi^n + \psi_{n-1}\varphi^{n-1} + \cdots + \psi_1\varphi + \psi_0 = 0.$$Let $a \in E$. Since $\IntR(E,D) \subseteq K[x]_{(x-a)}$, which is integrally closed, then $\varphi$ being integral over $\IntR(E,D)$ means that $\varphi$ is integral over $K[x]_{(x-a)}$. This shows that $\varphi \in K[x]_{(x-a)}$. Then we have $$\varphi(a)^n + \psi_{n-1}(a)\varphi(a)^{n-1} + \cdots + \psi_1(a)\varphi(a) + \psi_0(a) = 0.$$Note that $\varphi(a)$ is defined and in $K$. The above equation shows that $\varphi(a)$ is integral over $D$ as each $\psi_i(a) \in D$. Thus, $\varphi(a) \in D$. This holds for all $a \in E$, meaning $\varphi \in \IntR(E,D)$, which shows that $\IntR(E,D)$ is integrally closed. 
\end{proof}

We need another lemma to aide us. This lemma constructs an integer-valued rational function valued in a valuation domain. The constructed rational function has the flavor of a function guaranteed by continuity, such as by Proposition 2.1 of \cite{Liu}, but there are additional properties to give more control. 

\begin{lemma}\label{Lem:Z}
	Let $V$ be a valuation domain with value group $\Gamma$, maximal ideal $\m$, associated valuation $v$, and field of fractions $K$. Suppose that $\Gamma$ is not divisible or $V/\m$ is not algebraically closed. Then for all $\varepsilon, \delta \in \Gamma$, with $\varepsilon, \delta > 0$, and $c \in K$, there exist $\varphi \in \IntR(K,V)$ and $\gamma > \delta$ with the following properties:
	\begin{enumerate}[i)]
		\item $v(\varphi(d)) \leq \varepsilon$ for all $d \in K$;
		\item $v(\varphi(d)) > 0$ if and only if $v(d-c) \geq \gamma$; and
		\item $v(\varphi(c)) = \varepsilon$. 
	\end{enumerate}
	
\end{lemma}

\begin{proof} 
	We split into two cases.
	\begin{enumerate}
		\item Suppose that $\Gamma$ is not divisible.
		
		Let $\varepsilon, \delta \in \Gamma$ with $\varepsilon, \delta > 0$ and $c \in K$. Since $\Gamma$ is not divisible, there exists $\alpha \in \Gamma$ and $m \in \N$ with $m > 0$ such that $\frac{\alpha}{m}\notin \Gamma$. If $\frac{\varepsilon}{m} \in \Gamma$, then $\frac{\alpha}{m} + \frac{\varepsilon}{m} = \frac{\alpha + \varepsilon}{m} \notin \Gamma$ and take $a \in K$ such that $v(a) = \alpha + m\delta + \varepsilon$, $b \in K$ such that $v(b) = \alpha + m\delta$, and $n = m$. If $\frac{\varepsilon}{m} \notin\Gamma$, then $\frac{\varepsilon}{2m} \notin\Gamma$ and take $a \in K$ such that $v(a) = 2\varepsilon + 2m\delta$, $b \in K$ such that $v(b) = \varepsilon + 2m\delta$, and $n = 2m$. In either case, $\frac{v(a)}{n}$ and $\frac{v(b)}{n}$ are both not in $\Gamma$.
		
		Now set $$\varphi(x) = \frac{(x-c)^n+a}{(x-c)^n+b}.$$Let $d \in K$. We check that \[v(\varphi(d)) = \begin{cases}0, & \text{if $nv(d-c)<v(b)$},\\
			nv(d-c)-v(b), &\text{if $v(b) < nv(d-c) < v(a)$}, \\
			v(a)-v(b),& \text{if $nv(d-c) > v(a)$.} \end{cases} \]This shows that $\varphi \in \IntR(K,V)$ since $\frac{v(a)}{n}, \frac{v(b)}{n} \notin \Gamma$. Furthermore, we see that $nv(d-c) < v(a)$ implies $nv(d-c)-v(b) < v(a) - v(b) = \varepsilon$, so $v(\varphi(d)) \leq \varepsilon$ for all $d \in K$. This also shows that $v(\varphi(d)) > 0$ if and only if $v(d-c) \geq \frac{v(b)}{n}$. Note that $\frac{v(b)}{n} > \delta$. We also verify that $v(\varphi(c)) = \varepsilon$. 
		
		\item Suppose that $\Gamma$ is divisible and $V/\m$ is not algebraically closed. 
		
		Pick $b \in K$ such that $v(b) > \delta$. Then let $a \in K$ such that $v(a) = v(b) + \frac{\varepsilon}{n}$, which exists since $\Gamma$ is divisible. Because $V/\m$ is not algebraically closed, there exists a monic, nonconstant, unit-valued polynomial $f \in V[x]$. Now set $$\varphi(x) = \frac{a^nf(\frac{x-c}{a})}{b^nf(\frac{x-c}{b})},$$where $n$ is the degree of $f$. Now let $d \in K$. We use the fact that $v(a_1^n f(\frac{a_2}{a_1})) = \min\{nv(a_1), nv(a_2)\}$ for each $a_1, a_2 \in K$ with $a_1 \neq 0$ \cite[Corollary 2.3]{PruferNonDRings} to calculate that $$v(\varphi(d)) = \begin{cases}0, &\text{if $v(d-c) < v(b)$}\\
			nv(d-c)-nv(b), &\text{if $v(b) \leq v(d-c) \leq v(a)$}\\
			nv(a) - nv(b), &\text{if $v(a) < v(d-c)$}.\end{cases}$$This shows that $\varphi \in \IntR(K,V)$ and $v(\varphi(d)) \leq nv(a)-nv(b) = \varepsilon$ for all $d \in K$ and $v(\varphi(d)) = nv(a) - nv(b) > 0$ if and only if $v(d-c) \geq v(b)$. Note that $v(b) > \delta$. Lastly, we verify that $v(\varphi(c)) = nv(a)- nv(b) = \varepsilon$. 
	\end{enumerate}
\end{proof}

Now we are ready to link the property of being a Prüfer domain with the strong Skolem property in the context of rings of integer-valued rational functions. 

\begin{proposition}\label{Prop:tSkf}
	Let $V$ be a valuation domain with field of fractions $K$. Suppose that the value group of $V$ is not divisible or the residue field of $V$ is not algebraically closed. Also suppose that $E$ is some subset of $K$ that is strongly coherent for $\IntR(E,V)$. Then for all nonzero fractional ideals $I$ of $\IntR(E,V)$, we have $I_t = I^{\Sk_f}$.
\end{proposition}

\begin{proof}	
	It suffices to show that $I_v = I^{\Sk}$ for all nonzero finitely-generated integral ideals $I$ of $\IntR(E,V)$. 
	
	Suppose that $I$ is an integral ideal of $\IntR(E,V)$. Let $\varphi \in I^{\Sk}$ and $\psi \in I^{-1}$ be nonzero. Let $a \in E$. Then $\varphi(a) \in I(a)$ so there exists $\rho \in I$ such that $\rho(a) = \varphi(a)$. Since $\psi\rho \in \IntR(E,V)$, we have that $\psi(a)\rho(a) \in V$, except for possibly the finitely many values of $a \in E$ for which $\psi(a)$ is not defined. If $\psi(a)$ is defined, we have that $\psi(a)\varphi(a) = \psi(a)\rho(a) \in V$. Since $E$ is a strongly coherent set for $\IntR(E,V)$ and $(\psi\varphi)(a) \in V$ for all but finitely many $a \in E$, we get   $\varphi\psi \in \IntR(E,V)^{\Sk} = \IntR(E,V)$, implying that $\varphi \in (I^{-1})^{-1} = I_v$. We then have $I^{\Sk} \subseteq I_v$. 
	
	Now suppose that $I$ is finitely-generated, so there exist $\psi_1, \dots, \psi_n \in \IntR(E,V)$ such that $I=(\psi_1, \dots, \psi_n)$. Let $v$ be the valuation associated with $V$ and $\Gamma$ the value group. Suppose for a contradiction that there exists $\varphi \in I_v$ such that $\varphi \notin I^{\Sk}$. Then there exists $a \in E$ such that $v(\varphi(a)) < v(\psi_j(a))$ for all $j$. We know that $I(a)$ is generated by $\psi_i(a)$ for some $i$. We can assume that there does not exist a smallest strictly positive element in $\Gamma$ since otherwise $V$ has principal maximal ideal, implying that $\IntR(E,V)$ is Prüfer and therefore $I= I_v = I^{\Sk}$. Then there exists $\varepsilon \in \Gamma$ such that $0 < \varepsilon < v(\psi_i(a)) - v(\varphi(a))$. Due to \cite[Proposition 2.1]{Liu}, by the continuity of $\psi_1, \dots, \psi_n$, there exists $\delta \in \Gamma$ such that for all $d \in K$ with $v(d-a) > \delta$, we have $v(\psi_j(d)) > v(\psi_i(a)) - \varepsilon$ for all $j$. By Lemma \ref{Lem:Z}, there exists $\rho \in \IntR(E,V)$ such that $v(\rho(d)) \leq v(\psi_i(a)) - \varepsilon$ for all $d \in K$ and there exists $\gamma \in \Q\Gamma$ with $\gamma > \delta$ such that $v(\rho(d)) > 0$ if and only if $v(d-a) > \gamma$. 
	
	For each $j$, consider $\frac{\psi_j}{\varphi}$. Let $d \in E$. Then we have that $$v\left(\frac{\psi_j(d)}{\rho(d)}\right) = \begin{cases}
		v(\psi_j(d)), &\text{if $v(d-a)\leq \gamma$,}\\
		v(\psi_j(d)) - v(\rho(d)) \geq v(\psi_j(d)) - (v(\psi_j(d)) - \varepsilon) > 0, &\text{if $v(d-a) > \gamma$,}
	\end{cases}$$which shows that $(\psi_1, \dots, \psi_n) \subseteq (\rho)$. We then deduce that $\varphi \in I_v \subseteq (\rho)_v = (\rho)$ so we must have $v(\varphi(a)) \geq v(\rho(a))$. However,  $v(\rho(a)) = v(\psi_i(a)) - \varepsilon > v(\varphi(a))$, a contradiction. 
\end{proof}

\begin{remark}
	If $\IntR(E,V)$ does not have the strong Skolem property and $E$ is strongly coherent set for $\IntR(E,V)$, then the finitely-generated ideals of $\IntR(E,V)$ that are not Skolem closed are exactly the ones that are not divisorial. 
\end{remark}

Suppose $V$ is a valuation domain with algebraically closed residue field and maximal ideal that is not principal. Assume further that the value group of $V$ is not divisible. We show in the following proposition that $\IntR(V)$ does not have the strong Skolem property by explicitly finding a finitely-generated ideal that is not Skolem closed. This also shows that $\IntR(V)$ is not Prüfer in this case, since this finitely-generated ideal that is not Skolem closed is also not divisorial by Proposition \ref{Prop:tSkf}. Additionally, if $V$ is a valuation domain with algebraically closed residue field and divisible value group, then Proposition 2.38 of \cite{Liu} implies that $\IntR(V)$ is not Prüfer. This gives an alternative proof of Theorem 2.29 of \cite{Liu}, which classifies exactly when $\IntR(V)$ is a Prüfer domain for a valuation domain $V$.

\begin{proposition}\label{Prop:NotStrongSkolemValuation}
	Let $V$ be a valuation domain with algebraically closed residue field and maximal ideal that is not principal. Suppose the value group $\Gamma$ of $V$ is not divisible. Let $\m$ denote the maximal ideal of $V$. Then for any $t \in \m$, the finitely-generated ideal $(x^2, t^2)$ of $\IntR(V)$ is not Skolem closed, meaning $\IntR(V)$ does not have the strong Skolem property. 
\end{proposition} 

\begin{proof}
	We will show that the finitely-generated ideal $(x^2, t^2)$ of $\IntR(V)$ is not Skolem closed by showing that $tx \in (x^2,t^2)^{\Sk} \setminus (x^2, t^2)$. Let $d \in V$. Then the value ideal $(x^2,t^2)(d)$ is equal to $(d^2, t^2)$, so $$(x^2,t^2)(d) = \begin{cases}
		(d^2), &\text{if $v(d) \leq v(t)$}\\
		(t^2), &\text{if $v(d) > v(t)$.}
	\end{cases}$$On the other hand, we have that if $v(d) \leq v(t)$, then $v(td) \geq v(d^2)$, so $td \in (d^2)$. If $v(d) > v(t)$, then $v(td) > v(t^2)$, so $td \in (t^2)$. Thus, $tx \in (x^2, t^2)^{\Sk}$. 
	
	Now we want to show that $tx \notin (x^2, t^2)$. Suppose on the contrary that there exists $\varphi, \psi \in \IntR(V)$ such that $tx = \varphi(x)x^2 + \psi(x)t^2$. Let $d \in V$ such that $0 \leq v(d) < v(t)$. Then $td = \varphi(d)d^2 + \psi(d)t^2$ implies $v(td) = v(\varphi(d)d^2)$ since $v(td) < v(t^2) \leq v(\psi(d)t^2)$. This implies then that $v(\varphi(d)) = v(t) - v(d)$. Thus, $\minval_\varphi(\gamma) = v(t) - \gamma$ for $\gamma \in \Gamma$ such that $0 \leq \gamma < v(t)$. This also shows that $\minval_\varphi(v(t)) = 0$ and therefore there exists some $\varepsilon \in \Q\Gamma$ with $\varepsilon > 0$ such that there exists $c_1, c_2 \in \Z$ and $\beta_1, \beta_2 \in \Gamma$ so that $$\minval_\varphi(\gamma) = \begin{cases}
		c_1\gamma + \beta_1, &\text{if $v(t) - \varepsilon < \gamma \leq v(t)$}\\
		c_2\gamma + \beta_2, &\text{if $v(t) \leq \gamma < v(t) + \varepsilon$,}
	\end{cases}$$ since $\minval_\varphi$ is piecewise linear \cite[Proposition 2.16]{Liu}. Note that we found that $c_1 = -1$ and we must have $c_2 \geq 0$ since $\minval_\varphi(v(t)) = 0$. Therefore, by \cite[Lemma 2.25]{Liu}, we have that $\varphi \notin \IntR(V)$, a contradiction. This means that $tx \notin (x^2, t^2)$. 
	
	Since $(x^2, t^2)$ is a finitely-generated ideal of $\IntR(V)$ that is not Skolem closed, we know that the ring $\IntR(V)$ does not have the strong Skolem property. 
\end{proof}

\bibliographystyle{amsalpha}
\bibliography{references}
\end{document}